\documentclass[12pt]{amsart}
\usepackage{amssymb}
\usepackage{verbatim}
\usepackage{hyperref}
\usepackage{color}

\begin{document}

\newtheorem{theorem}{Theorem}    
\newtheorem{proposition}[theorem]{Proposition}
\newtheorem{conjecture}[theorem]{Conjecture}
\def\theconjecture{\unskip}
\newtheorem{corollary}[theorem]{Corollary}
\newtheorem{lemma}[theorem]{Lemma}
\newtheorem{sublemma}[theorem]{Sublemma}
\newtheorem{fact}[theorem]{Fact}
\newtheorem{observation}[theorem]{Observation}
\theoremstyle{definition}
\newtheorem{definition}{Definition}
\newtheorem{notation}[definition]{Notation}
\newtheorem{remark}[definition]{Remark}
\newtheorem{question}[definition]{Question}
\newtheorem{questions}[definition]{Questions}
\newtheorem{example}[definition]{Example}
\newtheorem{problem}[definition]{Problem}
\newtheorem{exercise}[definition]{Exercise}

\numberwithin{theorem}{section}
\numberwithin{definition}{section}
\numberwithin{equation}{section}

\def\reals{{\mathbb R}}
\def\torus{{\mathbb T}}
\def\integers{{\mathbb Z}}
\def\naturals{{\mathbb N}}
\def\complex{{\mathbb C}\/}
\def\distance{\operatorname{distance}\,}
\def\support{\operatorname{support}\,}
\def\dist{\operatorname{dist}\,}
\def\Span{\operatorname{span}\,}
\def\degree{\operatorname{degree}\,}
\def\dim{\operatorname{dim}\,}
\def\codim{\operatorname{codim}}
\def\trace{\operatorname{trace\,}}
\def\Span{\operatorname{span}\,}
\def\dimension{\operatorname{dimension}\,}
\def\codimension{\operatorname{codimension}\,}
\def\nullspace{\scriptk}
\def\kernel{\operatorname{Ker}}
\def\Re{\operatorname{Re\,} }
\def\Im{\operatorname{Im\,} }
\def\eps{\varepsilon}
\def\lt{L^2}
\def\diver{\operatorname{div}}
\def\curl{\operatorname{curl}}
\def\expect{\mathbb E}
\def\bull{$\bullet$\ }
\def\det{\operatorname{det}}
\def\Det{\operatorname{Det}}

\newcommand{\norm}[1]{ \|  #1 \|}
\newcommand{\Norm}[1]{ \Big\|  #1 \Big\| }
\newcommand{\set}[1]{ \left\{ #1 \right\} }
\def\one{{\mathbf 1}}
\newcommand{\modulo}[2]{[#1]_{#2}}
\newcommand{\abr}[1]{ \langle  #1 \rangle}

\def\bp{\mathbf p}
\def\bff{\mathbf f}
\def\bg{\mathbf g}

\def\scriptf{{\mathcal F}}
\def\scriptg{{\mathcal G}}
\def\scriptm{{\mathcal M}}
\def\scriptb{{\mathcal B}}
\def\scriptc{{\mathcal C}}
\def\scriptt{{\mathcal T}}
\def\scripti{{\mathcal I}}
\def\scripte{{\mathcal E}}
\def\scriptv{{\mathcal V}}
\def\scriptw{{\mathcal W}}
\def\scriptu{{\mathcal U}}
\def\scriptS{{\mathcal S}}
\def\scripta{{\mathcal A}}
\def\scriptr{{\mathcal R}}
\def\scripto{{\mathcal O}}
\def\scripth{{\mathcal H}}
\def\scriptd{{\mathcal D}}
\def\scriptl{{\mathcal L}}
\def\scriptn{{\mathcal N}}
\def\scriptp{{\mathcal P}}
\def\scriptk{{\mathcal K}}
\def\scriptP{{\mathcal P}}
\def\scriptj{{\mathcal J}}
\def\frakv{{\mathfrak V}}
\def\frakG{{\mathfrak G}}
\def\frakA{{\mathfrak A}}

\author{Marcos Charalambides}
\address{
        Marcos Charalambides\\
        Department of Mathematics\\
        University of California \\
        Berkeley, CA 94720-3840, USA}
\email{marcos@math.berkeley.edu}

\author{Michael Christ}
\address{
        Michael Christ\\
        Department of Mathematics\\
        University of California \\
        Berkeley, CA 94720-3840, USA}
\email{mchrist@math.berkeley.edu}
\thanks{Research of both authors supported in part by NSF grant DMS-0901569. The second
author was also supported by the Mathematical Sciences Research Institute.}


\date{December 13, 2011}

\title
[Near-extremizers of Young's inequality]
{Near-extremizers of Young's Inequality \\ For Discrete Groups}

\begin{abstract}
Those functions which nearly extremize Young's convolution inequality are characterized for
discrete groups which have no nontrivial finite subgroups.
Near-extremizers of the Hausdorff-Young inequality are characterized for $\integers^d$.
\end{abstract}

\maketitle

\section{Introduction}

Let $G$ be any discrete group. 
We denote the product of $x,y\in G$ by $x+y$, even though $G$ is not assumed to be Abelian. 
Equip $G$ with counting measure, so that $\norm{f}_{L^p(G)} =\big(\sum_{x\in G}|f(x)|^p \big)^{1/p} $
for any $p\in[1,\infty)$, while $\norm{f}_\infty = \sup_{x\in G}|f(x)|$.

Write
\[ \langle f_1*f_2,f_3\rangle = \sum_{x,y\in G} f_1(x)f_2(y)f_3(x+y).  \]
Young's convolution inequality states that for any discrete group,
\begin{equation} \label{young}
\Big| \langle f_1*f_2,f_3\rangle \Big| \le \prod_{j=1}^3\norm{f_j}_{p_j} \end{equation}
for complex-valued functions whenever each $p_j\in[1,\infty]$ and $\sum_j p_j^{-1}=2$.
Equivalently, $\norm{f_1*f_2}_q\le\norm{f_1}_{p_1}\norm{f_2}_{p_2}$
for $q^{-1}=p_1^{-1}+p_2^{-1}-1$ provided $q\in[1,\infty]$.

The optimal constants in these inequalities equal $1$, since equality is attained
if each $f_j$ is supported on a single point $z_j$, provided that $z_3=z_1+z_2$.
If $G=\integers^d$ for some $d\ge 1$,
then it is an elementary fact that so long as all $p_j$ belong to $(1,\infty)$,
equality is attained only if the support of each $f_j$ is a singleton, 
with the relation $z_3=z_1+z_2$ between their respective supports $\set{z_j}$.
This last assertion does not hold for general groups, for if $H$ is a finite
subgroup of $G$ and each function $f_j$ is the indicator function of 
$H$, then equality holds in \eqref{young}. 

Another relevant example is as follows.
Let $G=\integers$. Let $N$ be a large positive integer. Define $f_j^{(N)}$ to be the indicator
function of the interval $[-N,N]$ for $j=1,2,3$. 
Then $\prod_j\norm{f_j^{(N)}}_{p_j}=(2N+1)^2 = 4N^2+O(N)$ provided that $\sum_j p_j^{-1}=2$,
while $\langle f_1^{(N)}*f_2^{(N)},f_3^{(N)}\rangle = 3N^2+O(N)$.
Thus as $N\to \infty$,
\[ |\langle f_1^{(N)}*f_2^{(N)},f_3^{(N)}\rangle | 
\ge \Big(\frac34-O(N^{-1})\Big) \prod_{j=1}^3\norm{f_j^{(N)}}_{p_j}. \]
Yet any subsequence of $f_1^{(N)}/\norm{f_1^{(N)}}_{p_1}$, even modulo translations, converges weakly to $0$.
The main result of this paper implies that such a phenomenon cannot arise if
$\frac34$ is replaced by a constant sufficiently  close to $1$.

\begin{definition}
Let exponents $p_j\in(1,\infty)$ satisfy $\sum_{j=1}^3 p_j^{-1}=2$.
A triple of nonzero functions $(f_1,f_2,f_3)\in \times_{j=1}^3 L^{p_j}$
is a $(1-\delta)$--near extremizer for Young's inequality if
\begin{equation} |\langle f_1*f_2,f_3\rangle| \ge (1-\delta)\prod_{j=1}^3\norm{f_j}_{p_j}.  \end{equation}
\end{definition}

\begin{definition}
A group is said to be torsion-free if it has no finite subgroups except the trivial
subgroup of cardinality equal to one.
\end{definition}

\begin{theorem} \label{thm:main}
For any $d\ge 1$ and any 
ordered triple of exponents $p_j\in(1,\infty)$ satisfying $\sum_j p_j^{-1}=2$,
there exists a function $\delta:(0,1]\to(0,\infty)$
satisfying $\delta(\eps)\to 0$ as $\eps\to 0$ with the following property.
Let $G$ be any torsion-free discrete group. 
For any $\eps\in(0,1]$, if a triple $(f_1,f_2,f_3)$ is a
$(1-\delta(\eps))$--near extremizer for Young's inequality for $G$,
then there exist points $z_j\in G$ and scalars $c_j\in\complex$
such that for each index $j\in\set{1,2,3}$,
\begin{equation} \norm{f_j-c_j\one_{\set{z_j}}}_{p_j}<\eps\norm{f_j}_{p_j}.  \end{equation}
Moreover $z_3=z_1+z_2$.
\end{theorem}

Equivalently, $\norm{f_j}_\infty \ge (1-\eps)\norm{f_j}_{p_j}$, for a different but comparable value of $\eps$.
A more quantitative formulation is as follows. 
\begin{theorem} \label{thm:quantitative}
Let $p_1,p_2,q\in (1,\infty)$ satisfy $q^{-1}=p_1^{-1} + p_2^{-1}-1$.
There exist a continous nondecreasing function $\Lambda:(0,1]\to(0,1]$,
an exponent $\gamma>0$, and a constant $c>0$ satisfying \[\Lambda(t) \le 1-c(1-t)^\gamma \text{ as } t\to 1^-\]
such that for any torsion-free discrete group $G$ and for any functions $f_i\in L^{p_i}(G)$,
\begin{equation}
\norm{f_1*f_2}_{q} \le \norm{f_1}_{p_1} \norm{f_2}_{p_2}
\cdot \Lambda\Big(\frac{\norm{f_1}_{\infty}}{\norm{f_1}_{p_1}}\Big)
\cdot \Lambda\Big(\frac{\norm{f_2}_{\infty}}{\norm{f_2}_{p_2}}\Big).
\end{equation} \end{theorem}

The analogue of Theorem~\ref{thm:main} for $\reals^d$, has been proved in \cite{christrealyoung}.
The simple method of proof developed here cannot apply to $\reals^d$. 
Nonetheless, the questions for $\reals$ and for $\integers$ are connected; the following reasoning shows
that if Theorem~\ref{thm:main} were not valid for $\integers$, then its analogue for $\reals$ would
also necessarily fail. 
Indeed, fix an extremizer $F\in L^p(\reals)$ for Young's inequality for $\reals$.
If a sequence of functions $f_\nu\in L^p(\integers)$ were to furnish a counterexample
to Theorem~\ref{thm:main} for $\integers$, then the sequence of functions
$g_\nu(x) = \sum_{n\in\integers} f_\nu(n) \nu^{1/p} F(\nu x)$ would provide a counterexample for $\reals$. 

The Hausdorff-Young inequality for $\integers^d$ states that
\begin{equation} \norm{\widehat{f}}_{L^{p'}(\torus^d)}\le \norm{f}_{\ell^p(\integers^d)} \end{equation}
for all $1\le p\le 2$, where $p'=p/(p-1)$,
$\torus^d = \reals^d/\integers^d$, equipped with the measure which pulls back
to Lebesgue measure on $\reals^d$ under the natural projection of $\reals^d$ onto $\torus^d$,
$\widehat{f}(\xi) = \sum_{x\in\integers^d} e^{-2\pi i \xi\cdot x}f(x)$,
and
$\norm{g}_{L^q(\torus^d)} = \big((2\pi)^{-d}\int_{\torus^d} |g(y)|^q\,dy \big)^{1/q}$.

Our results for Young's inequality have the following implication for the Hausdorff-Young inequality.
\begin{theorem}
Let $p\in(1,2)$. 
For any $\eps>0$ there exists $\delta>0$ such that for any dimension $d$
and any $(1-\delta)$--near extremizer $f\in L^p(\integers^d)$ for the Hausdorff-Young inequality,
there exists $z\in \integers^d$ such that $\norm{f}_{L^p(\integers^d\setminus\set{z})}<\eps\norm{f}_p$.
More precisely, there exist a continuous nondecreasing function $\Lambda:(0,1]\to(0,1]$,
an exponent $\gamma\in(0,\infty)$, 
and a constant $c>0$ satisfying \[\Lambda(t) \le 1-c(1-t)^\gamma \text{ as } t\to 1^-\]
such that for any $d\ge 1$ and for any function $f\in \ell^p(\integers^d)$,
\begin{equation}
\norm{\widehat{f}}_{p'} \le \norm{f}_p\cdot \Lambda\Big(\frac{\norm{f}_{\infty}}{\norm{f}_{p}}\Big).
\end{equation}\end{theorem}

\section{Lorentz space bound, and a consequence} \label{section:lorentz}

We will work with triples of exponents $\bp=(p_j: 1\le j\le 3)$ which satisfy the two hypotheses
\begin{align} &\sum_j p_j^{-1}=2 \label{H1}
\\ &p_j\in(1,\infty) \ \text{ for all } j\in \set{1,2,3}.  \label{H2} \end{align}

Let $G$ be any discrete group.
Let $(f_j)_{1\le j\le 3}$ be a triple of functions with $f_j\in L^{p_j}$. 
If $|\langle f_1*f_2,f_3\rangle|\ge (1-\delta)\prod_{j=1}^3\norm{f_j}_{p_j}$,
then the same holds with each $f_j$ replaced by $|f_j|$. 
If $|f_j|$ satisfies the desired conclusion, then so does $f_j$.
Therefore it suffices to work only with nonnegative functions in the sequel.

Recall the Lorentz spaces $L^{p,r}$, where $p\in(1,\infty)$ and $r\in[1,\infty]$ \cite{steinweiss}.
If $f=\sum_{k\in\integers} 2^k F_k$
where $\one_{E_k}\le F_k\le 2\one_{E_k}$ and the sets $E_k$ are pairwise disjoint,
then the Lorentz norm $\norm{f}_{L^{p,r}}$ satisfies
\[ \norm{f}_{L^{p,r}} \asymp \Big(\sum_{k=-\infty}^\infty (2^k|E_k|^{1/p})^r\Big)^{1/r} \]
in the sense that one of these quantities is finite if and only if the other is
finite, and each is majorized by a constant multiple of the other, 
where these constants depend only on the indices $p,r$.
In particular, if $p<r<\infty$ then there exist $\eta\in(0,1)$ 
and $C<\infty$ such that for any $f\in L^{p,r}$,
\[ \norm{f}_{p,r} \le C\norm{f}_p^{1-\eta} \big(\sup_k 2^k|E_k|^{1/p}\big)^\eta.  \]

By real interpolation \cite{steinweiss}, for any $\bp$ satisfying \eqref{H1} and \eqref{H2},
there exist exponents $r_j>p_j$ such that
\begin{equation} \label{younglorentz}
\langle f_1*f_2,f_3\rangle 
\le C\prod_{j=1}^3\norm{f_j}_{L^{p_j,r_j}}
\end{equation}
where $C<\infty$ depends only on $\bp$.


By H\"older's inequality,
\eqref{younglorentz} implies that there exists $\eta=\eta(p_1,p_2,p_3)>0$
such that for all $\set{f_j}$,
\begin{equation} \label{younglorentz2}
\langle f_1*f_2,f_3\rangle 
\le C\prod_{j=1}^3\norm{f_j}_{L^{p_j}}^{1-\eta}
\prod_{j=1}^3 \sup_{k_j\in\integers} (2^{k_j}|E_{j,k_j}|^{1/p_j})^\eta.
\end{equation}
Therefore
\begin{lemma} \label{lemma:kappa}
Let $\bp$ satisfy \eqref{H1} and \eqref{H2}.
For any $\delta>0$ there exists $\rho>0$
such that for any nonnegative functions $f_j$ satisfying
$\langle f_1*f_2,f_3\rangle\ge\delta\prod_{j=1}^3 \norm{f_j}_{p_j}$,
for each $j\in\set{1,2,3}$ there exists $\kappa_j\in\integers$ such that
\begin{equation} \label{kappaproperty}
2^{\kappa_j}|E_{j,\kappa_j}|^{1/p_j} \ge \rho \norm{f_j}_{p_j}.  \end{equation}
\end{lemma}

\section{A consequence of a proof of Young's inequality} \label{section:review}

One proof of Young's inequality uses H\"older's inequality to reduce matters to
the simpler inequality $\norm{f*g}_p\le \norm{f}_1\norm{g}_p$. 
In this section we review this proof  in order to  extract additional information from it.


\begin{lemma} \label{lemma:L1reduction}
Let $G$ be any discrete group.
For any $\bp$ satisfying hypotheses \eqref{H1} and \eqref{H2},
there exist exponents $s_1,s_2\in (1,\infty)$ 
and a constant $C<\infty$
such that for any $\delta\in[0,1]$, if $\bff=(f_j: 1\le j\le 3)$ is any 
$(1-\delta)$--near extremizer for Young's inequality for $G$ with exponents $\bp$,
then $\bg=(|f_1|^{p_1/s_1},|f_2|^{p_2/s_2},|f_3|^{p_3})$
is a $(1-C\delta)$--near extremizer for Young's inequality with exponents $(s_1,s_2,1)$. 

Moreover, the same conclusion holds for any permutation of the three indices.
\end{lemma}

\begin{proof}
Without loss of generality, we may assume that each function $f_j$ is nonnegative.
Let $q=p'_3$ be the exponent conjugate to $p_3$,
set $\theta = p_1/q\in (0,1)$ and $\phi =  p_2/q\in (0,1)$, and write the cubic form as
\begin{align*} \langle f_1*f_2,f_3\rangle
&= \sum_{x,y} f_1(x)f_2(y)f_3(x+y)
\\
&= \sum_{x,y} \Big(f_1(x)^{1-\theta} f_2(y)^{1-\phi} f_3(x+y)\Big) \cdot 
\Big(f_1(x)^{\theta}f_2(y)^{\phi}\Big)
\\
&\le \Big(\sum_{x,y} 
f_1(x)^{(1-\theta)p_3} f_2(y)^{(1-\phi)p_3}f_3(x+y)^{p_3}\Big)^{1/p_3} \ \cdot\ 
\Big(\sum_{x,y} f_1(x)^{p_1} f_2(y)^{p_2} \Big)^{1/q}
\\
&= \Big(\sum_{x,y} f_1(x)^{r_1} f_2(y)^{r_2}f_3(x+y)^{p_3}\Big)^{1/p_3} 
\norm{f_1}_{p_1}^{p_1/q} \norm{f_2}_{p_2}^{p_2/q}
\end{align*}
where $r_1=(q-p_1)p_3/q$ and $r_2=(q-p_2)p_3/q$.  Define $s_j=p_j/r_j$ for $j=1,2$, 
and $s_3=1$.
Then $f_3^{p_3}\in L^1$, while $g_j=f_j^{r_j}\in L^{s_j}$.  Moreover
\begin{multline*} \frac{1}{s_1}+\frac{1}{s_2} = \frac{r_1}{p_1}+\frac{r_2}{p_2}
= \frac{(q-p_1)p_3}{qp_1} +\frac{(q-p_2)p_3}{qp_2} \\
= p_3\Big( \frac{1}{p_1} -\frac{1}{q} +\frac{1}{p_2} -\frac{1}{q} \Big)
= p_3\Big( \frac{1}{p_1} +\frac{1}{p_2} -2 + \frac{2}{p_3} \Big)
= 1 \end{multline*}
since $\sum_{j=1}^3 \frac1{p_j}=2$.
Therefore the ordered triple $(s_1,s_2,1)$ satisfies \eqref{H1},
and Young's inequality with general exponents $\bp$ now follows directly from the special case
with exponents $(s_1,s_2,1)$. 

If $\bff$ is a $(1-\delta)$--near extremizer for the exponents $\bp$, then
\[
(1-\delta)\prod_{j=1}^3 \norm{f_j}_{p_j}
\le \langle f_1*f_2,f_3\rangle
\le \norm{f_1}_{p_1}^{p_1/q} \norm{f_2}_{p_2}^{p_2/q} 
\Big(\sum_{x,y} f_1(x)^{r_1} f_2(y)^{r_2}f_3(x+y)^{p_3}\Big)^{1/p_3}
\]
and therefore
\[ \sum_{x,y} f_1(x)^{r_1} f_2(y)^{r_2}f_3(x+y)^{p_3} \ge (1-\delta)^{p_3} 
\norm{f_1}_{p_1}^{r_1} \norm{f_2}_{p_2}^{r_2} \norm{f_3}_{p_3}^{p_3}.  \]
Equivalently,
$\langle g_1*g_2,g_3\rangle \ge (1-\delta)^{p_3} \prod_j\norm{g_j}_{s_j}$.  
\end{proof}


\section{A consequence of strict uniform convexity}

Extra information will later be derived from a study of near-extremizers
in the special case when exactly one of the three indices $p_j$ equals $1$ ---
paradoxically, a case in which our main conclusions do not hold. 
In this section we exploit the strict uniform convexity of the unit ball of $L^p$
for $p\in(1,\infty)$ to characterize near-extremizers $(f_j: j\in\integers)$ 
of the triangle inequality $\norm{\sum_j f_j}_p\le \sum_j \norm{f_j}_p$.

Among the well-known inequalities of Clarkson \cite{clarkson} are the following two.
\begin{lemma} [Clarkson's inequalities \cite{clarkson}]
Let $p\in(1,\infty)$ and set $q=p'=p/(p-1)$. Let $f,g\in\ell^p$. 
If $p\ge 2$ then
\begin{equation*}
\norm{f+g}_p^p+\norm{f-g}_p^p \le 2^{p-1}\norm{f}_p^p+2^{p-1}\norm{g}_p^p.
\end{equation*}
If $p\le 2$ then
\begin{equation*}
\norm{f+g}_p^q + \norm{f-g}_p^q
\le 2(\norm{f}_p^p+\norm{g}_p^p)^{q-1}. 
\end{equation*}
\end{lemma}

We will need the following direct consequence.
\begin{corollary} \label{cor:phew}
Let $p\in(1,\infty)$. Then there exist $C,\gamma\in\reals^+$ such that
for any $f,g\in L^p$ and any $\eps\in[0,1]$, if 
\begin{equation}
\norm{f+g}_p^p\ge (1-\eps) \big(2^{p-1}\norm{f}_p^p + 2^{p-1}\norm{g}_p^p\big)
\end{equation}
then
\begin{equation}
\norm{f-g}_p^p\le C\eps^\gamma \big(\norm{f}_p^p + \norm{g}_p^p\big).
\end{equation}
\end{corollary}

\begin{lemma} \label{lemma:convexity}
For any $p\in(1,\infty)$, there exist $C,\rho\in\reals^+$
with the following property. Let $S$ be a nonempty countable index set.
Let $\set{f_i: i\in S}$ be a set of $L^p$ functions,
at least one of which has positive norm.  If $\sum_{i\in S}\norm{f_i}_p<\infty$ and
\begin{equation}
\norm{\sum_{i\in S} f_i}_p \ge (1-\delta)\sum_{i\in S} \norm{f_i}_p, \end{equation}
then the set $S$ can be partitioned as a disjoint union $S =S'\cup S''$
so that the functions $f_i$ and their sum $F=\sum_{i\in S} f_i$  satisfy
\begin{align} \label{triangle1}
\sum_{i\in S''} \norm{f_i}_p &\le C\delta^\rho \sum_{i\in S} \norm{f_i}_p
\\
\label{triangle2}
\Norm{f_i-\frac{\norm{f_i}_p}{\norm{F}_p} F}_p &\le C\delta^\rho \norm{f_i}_p
\qquad \text{ for all } i\in S'.  \end{align}
\end{lemma}

\begin{proof}
We may suppose without loss of generality that $\norm{F}_p=1$.
There exists a bounded linear functional $\ell\in (L^p)^*$
which satisfies $\norm{\ell}_{(L^p)^*}=1$ and $\ell(F)=1$.
By discarding all those indices for which $\norm{f_i}_p=0$, 
we may assume that $\norm{f_i}_p>0$ for every $i\in S$.
Set $F_i=f_i/\norm{f_i}_p$.  For each $i\in S$ decompose
\[ F_i = s_i F+r_i \]
where $s_i=\ell(F_i)\in[-1,1]$ and $r_i=F_i-\ell(F_i)F= F_i-s_i F$.

Now
\begin{equation} \label{ribound}
\norm{r_i}_p=\norm{F_i-s_iF}_p\le C(1-s_i)^\gamma
\end{equation}
where $C,\gamma\in\reals^+$ depend only on $p$.
For
\[ \norm{F_i+s_iF}_p\ge \ell(F_i+s_iF) =2s_i, \]
so 
\[ 2^{p-1}\norm{F_i}_p^p +2^{p-1}\norm{s_i F}_p^p -\norm{F_i+s_iF}_p^p
\le 2^{p-1}+2^{p-1}-(2|s_i|)^p =2^p(1-|s_i|).  \] 
\eqref{ribound} thus follows from Corollary~\ref{cor:phew} with 
$f=F_i$ and $g=s_i F$.

Also
\[ 1 = \ell(F) = \sum_{i\in S} \ell(f_i) = \sum_{i\in S} s_i\norm{f_i}_p \]
while
\[ \sum_{i\in S} \norm{f_i}_p\le(1-\delta)^{-1}\norm{\sum_{i\in S} f_i}_p 
=(1-\delta)^{-1}\norm{F}_p \le 1+C\delta, \]
so
\[ \sum_{i\in S} (1-s_i)\norm{f_i}_p\le C\delta.  \]
Therefore
\[ \sum_{i\in S} \norm{r_i}_p^{1/\gamma} \norm{f_i}_p\le C\delta, \]
where $C\in\reals^+$ is a constant which depends only on $p$.
Substituting $\norm{r_i}_p = \norm{f_i}_p^{-1} \norm{f_i-\ell(f_i)F}_p$ gives
\[ \sum_{i\in S} \Big(\frac{\norm{f_i-\ell(f_i)F}_p}{\norm{f_i}_p}\Big)^{1/\gamma} \norm{f_i}_p\le C\delta.  \]

Let $\eta>0$ be a quantity to be chosen below, and define
\[ S''=\set{i\in S: \norm{f_i-\ell(f_i)F}_p\ge \eta\norm{f_i}_p } \]
and of course $S'=S\setminus S''$.
Then 
\[ \sum_{i\in S''} \eta^{1/\gamma} \norm{f_i}_p\le C\delta, \]
so
\[ \sum_{i\in S''} \norm{f_i}_p\le C\delta \eta^{-1/\gamma}. \]

On the other hand,
for every $i\in S'$ we have
$\norm{f_i-\ell(f_i)F}_p< \eta\norm{f_i}_p$. 
Now
\[
\norm{f_i-\norm{f_i}_pF}_p
\le
\norm{f_i-\ell(f_i)F}_p + |\ell(f_i)-\norm{f_i}_p|
\]
The relation $F_i = s_i F+r_i$ can be rewritten as
$f_i = \ell(f_i)F + r_i\norm{f_i}_p$,
so
\[
\big|\norm{f_i}_p-|\ell(f_i)|\big|=\big|\norm{f_i}_p-\norm{\ell(f_i)F}_p\big|
\le \norm{r_i}_p\norm{f_i}_p =\norm{f_i-\ell(f_i)F}_p.
\]
If $i\in S'$ then, by definition of $S'$, this is $\le \eta\norm{f_i}_p$. Therefore $ \norm{f_i-\norm{f_i}_pF}_p\le 2\eta\norm{f_i}_p$.

Thus both conclusions \eqref{triangle1},\eqref{triangle2} of the lemma hold, 
with $\eps(\delta)=\max(2\eta, C\delta \eta^{-1/\gamma})$.
Choosing $\eta=\delta^{\gamma/(1+\gamma)}$ yields the required bounds.
\end{proof}

\section{Conclusion of proof}

All of the reasoning so far applies to any discrete group, with or without torsion.
The remainder of the analysis relies on the following generalization of the Cauchy-Davenport inequality
due to Kemperman \cite{kemp}:
For any two finite subsets $A,B$ of any torsion-free group,
\begin{equation} \label{CD} |A+B|\ge |A|+|B|-1.  \end{equation}
See also the alternative proof of Hamidoune \cite{hamidoune}.

The following quantity $\scriptn$ quantifies the extent to which a function fails to be concentrated on
any small set.
\begin{definition}
Let $p\in[1,\infty)$ and $\eta>0$.
For any nonzero $f\in L^p(G)$, $\scriptn(f,\eta,p)$ denotes
the largest integer $N$ such that for all subsets $B\subset G$,
\begin{equation} \label{Bimplication}
\norm{f}_{L^{p}(G\setminus B)}^{p}< \eta\norm{f}_{p}^{p} \ \Rightarrow\ |B|\ge N.
\end{equation} \end{definition}
For any nonzero function $f$, $\scriptn(f,\eta,p)$ is well-defined, and is $\ge 1$.
Moreover, $\scriptn(f,\eta,p)=1$ if and only if there exists $z$ such that
$\norm{f}_{L^p(G\setminus\set{z})}^p\le \eta \norm{f}_p^p$.

The main step in the proof of Theorem~\ref{thm:main} is encapsulated in the next lemma.
\begin{lemma} \label{lemma:doubling}
Let $\bp$ satisfy \eqref{H1} and \eqref{H2}.
For any $\eta>0$ there exists $\delta>0$ with the following property.
Let $G$ be any torsion-free discrete group.
Let $(f_1,f_2)\in (L^{p_1}\times L^{p_2})(G)$ be any $(1-\delta)$--near extremizer. 
If $\scriptn(f_2,\eta,p_2)\ge 2$  then
\begin{equation} \label{Aimplication}
\scriptn(f_1,\eta,p_1)\ge 2\scriptn(f_2,\eta,p_2).  \end{equation}
\end{lemma}

\begin{proof} 
Let $C,\rho$ be the constants which appear in Lemma~\ref{lemma:convexity}
for the exponent $s_1$, and suppose henceforth that $\delta$ is sufficiently small that $C\delta^\rho<\eta$.
Set $N=\scriptn(f_2,\eta,p_2)$.

For $u\in G$, let $\tau_u g(x)=g(x-u)$ denote the right translate of a function $g:G\to\complex$.
Since $(f_1,f_2)$ is a $(1-\delta)$--near extremizing pair, Lemma~\ref{lemma:L1reduction} says that
\[ \norm{\sum_{u\in G} f_2^{p_2}(u)\tau_u f_1^{r_1}}_{s_1}
\ge (1-C_0\delta)\norm{f_1}_{p_1}^{r_1}\norm{f_2}_{p_2}^{p_2}, \]
where $s_j,r_j$ are the exponents which appear in that lemma and the constant $C_0$ depends on $p_1,p_2$ alone.

Apply Lemma~\ref{lemma:convexity} 
to the collection of functions $x\to \tau_u f_1^{r_1}(x)f_2^{p_2}(u)$ indexed by $u\in G$. 
Since $C\delta^\rho<\eta$, the lemma gives
a function $\scriptf\in L^{s_1}$ and a partition $G=S'\cup S''$ of $G$ such that 
\begin{gather} \label{app1} \sum_{u\in S''} f_2^{p_2}(u) <\eta \norm{f_2}_{p_2}^{p_2}
\\
\intertext{and} 
\label{app2}
\norm{\tau_u f_1^{r_1}-\scriptf}_{s_1}< \eta \norm{f_1}_{p_1}^{r_1}
\text{ for all $u\in S'$.}
\end{gather}
According to the definition \eqref{Bimplication} of $N=\scriptn(f_2,\eta,p_2)$, 
\eqref{app1} implies that $|S'|\ge N$.
Inequality \eqref{app2} implies that there exists a set $T$, which is a translate of $S'$, such that
\begin{equation} \norm{\tau_u f_1^{r_1}-f_1^{r_1}}_{s_1}\le 2\eta\norm{f_1}_{p_1}^{r_1} 
\ \text{ for all $u\in T$.} \end{equation}
In particular, $|T|=|S'|\ge N$.

Let $nT$ denote the set 
\[nT=\set{t_1+t_2+\cdots t_n: t_j\in T \text{ for all } j\in[1,n]}\]
of all $n$-fold sums of elements of $T$.
If $u\in nT$ then 
\[\norm{\tau_u f_1^{r_1}-f_1^{r_1}}_{s_1}\le 2n\eta\norm{f_1}_{p_1}^{r_1}.\]
Moreover, $|nT|\ge nN-n+1$ by the Cauchy-Davenport  inequality \eqref{CD}.
Since $N\ge 2$, $|nT|\ge \tfrac12 nN$. 

For any exponents $p\in [1,\infty)$ and $r\in[1,p]$
and any two nonnegative functions $g_i\in L^p$, 
$\norm{g_1-g_2}_{L^{p}}^p \le C \norm{g_1^r-g_2^r}_{L^{p/r}}
\big(\norm{g_1}_p+\norm{g_2}_p\big)^{p-r}$.
Therefore since $r_1\cdot s_1=p_1$,
\[\norm{\tau_u f_1-f_1}_{p_1}\lesssim n^{1/r_1}\eta^{1/r_1}\norm{f_1}_{p_1} 
\text{ for all $u\in nT$.} \]

Recall the decomposition $f_1=\sum_{k\in\integers} 2^k F_k$,
the associated sets $E_{k}$, and the index $\kappa=\kappa_1$ of Lemma~\ref{lemma:kappa}. 
Set $\tilde E_\kappa = E_{\kappa-1}\cup E_\kappa\cup E_{\kappa+1}$.
For any $u\in G$,
\[
\norm{\tau_u f_1-f_1}_{p_1}
\ge
\tfrac12 2^{\kappa} \,\big|\tau_u E_{\kappa} \setminus \tilde E_\kappa |^{1/p_1},
\]
where $A\setminus B$ denotes the intersection of a set $A$ with the complement of a set $B$.
One of the conclusions of Lemma~\ref{lemma:kappa} is that
$2^{\kappa}|E_{\kappa}|^{1/p_1}\asymp \norm{f_1}_{p_1}$
so long as $\delta$ does not exceed a certain positive constant which depends only on $\bp$,
not on $\eta$.  Therefore for all sufficiently small $\eta$, for every $u\in nT$,
\begin{equation} \label{oldtba}
\big|\tau_u E_{\kappa} \setminus \tilde E_\kappa| \le C(n\eta)^{p_1/r_1}|E_{\kappa}|. 
\end{equation}

The parameter $n$ has not yet been specified. Choose it to satisfy
$\tfrac14\le C(n\eta)^{p_1/r_1}\le\tfrac12$,
where $C$ is the constant in \eqref{oldtba}, assuming as we may that $\eta$ is sufficiently small.
Thus $n\asymp \eta^{-1}$.

For each $u\in nT$,
$\big|\tau_u E_{\kappa} \cap \tilde E_\kappa| \ge \tfrac12 |E_{\kappa}|$, that is,
\[ \int_{E_\kappa} \one_{\tilde E_\kappa}(\tau_{-u} x)\,dx\ge \tfrac12|E_\kappa|.  \]
Therefore
\[|E_\kappa|^{-1}\int_{E_\kappa} \sum_{u\in nT} \one_{\tilde E_\kappa}(\tau_{-u} x)\,dx\ge \tfrac12 |nT|.\]
Consequently there exists $x\in E_{\kappa}$ which satisfies
\[ \tau_{-u}(x)\in \tilde E_\kappa \text{ for at least $\tfrac12|nT|$ elements $u\in nT$};  \]
which is to say that 
$x+\tilde T\subset \tilde E_\kappa$ for some subset $\tilde T\subset nT$ of cardinality
$\ge \tfrac12 |nT|$.  We conclude that
\[ \big|\tilde E_\kappa| \ge \tfrac12|nT| \ge a\eta^{-1}N,  \]
where $a$ is a positive constant which depends only on $\bp$.

Let $A\subset G$ be any subset of cardinality $\le \tfrac12 a\eta^{-1}N$.  Then
\begin{align*}
\norm{f_1}_{L^{p_1}(G\setminus A)}
&\ge \norm{f_1}_{L^{p_1}(\tilde E_\kappa\setminus A)}
\\
&\ge 2^{\kappa-1} \big|\tilde E_\kappa \setminus A\big|^{1/p_1}
\\
&\ge 2^{\kappa-1}2^{-1/p_1} \big|\tilde E_\kappa \big|^{1/p_1}
\\
&\ge c_1 \norm{f_1}_{p_1}
\end{align*}
where $c_1>0$ is independent of $f_1,\eta$; the final inequality follows from \eqref{kappaproperty}.
The contrapositive statement is that for any $A$,
\begin{equation} \label{Aimplication0}
\norm{f_1}_{L^{p_1}(G\setminus A)}^{p_1}< c_1^{p_1}\norm{f_1}_{p_1}^{p_1} \ \Rightarrow\ |A| >
\tfrac{a}2\eta^{-1}N.\end{equation}
It is no loss of generality to assume that $\eta$ is smaller than any specified constant, 
in particular, that $\eta<c_1^{p_1}$ and $\tfrac{a}2 \eta^{-1}\ge 2$.
Thus
\[\norm{f_1}_{L^{p_1}(G\setminus A)}^{p_1}< \eta\norm{f_1}_{p_1}^{p_1} 
\ \Rightarrow\ |A|\ge 2N=2\scriptn(f_2,\eta,p_2),\]
concluding the proof of Lemma~\ref{lemma:doubling}.
\end{proof} 

\begin{proof}[Proof of Theorem~\ref{thm:main}]
The proof of Lemma~\ref{lemma:doubling} applies equally well with the roles of $f_1,f_2$ reversed. 
Therefore if $(f_1,f_2)$ is a $(1-\delta)$--near extremizer,
if $\scriptn(f_2,\eta,p_2)\ge 2$, and if $\delta$ is sufficiently small as a function of $\eta$
and $\bp$ alone, then
\begin{equation} \scriptn(f_2,\eta,p_2)\ge 
2\scriptn(f_1,\eta,p_1)\ge 4\scriptn(f_2,\eta,p_2),  \end{equation}
which is a contradiction.

Thus given $\eta$, if 
$(f_1,f_2)$ is a $(1-\delta)$--near extremizer for sufficiently small $\delta$,
then $\scriptn(f_2,\eta,p_2)= 1$.
The proof of Theorem~\ref{thm:main} is complete.
\end{proof}

Because $\delta$ was constrained only by the bound $\delta\le C\eta^\rho$ in this argument,
we have proved the following quantitative result, which establishes Theorem~\ref{thm:quantitative}.
\begin{proposition}
Let $G$ be any torsion-free discrete group.
Let $p_1,p_2\in(1,\infty)$ and suppose that $q^{-1}=p_1^{-1}+p_2^{-1}-1$
satisfies $q\in (1,\infty)$. There exist $c,\gamma \in\reals^+$
such that if $\norm{f_1*f_2}_q\ge (1-\delta)\norm{f_1}_{p_1}\norm{f_2}_{p_2}$, then 
\[ \frac{\norm{f_i}_\infty}{\norm{f_i}_{p_i}} \ge 1-c\delta^\gamma \]
for both indices $i=1,2$.
\end{proposition}

\section{The Hausdorff-Young inequality}

Let $f\in \ell^p(\integers^d)$. In the case $p\le\tfrac43$,
\begin{equation}\label{HausYoung}
\norm{\widehat{f}}_{p'}^2=\norm{(\widehat{f})^2}_{s'}=\norm{\widehat{f*f}}_{s'}\le\norm{f*f}_s\le\norm{f}_p^2
\cdot \Lambda\Big(\frac{\norm{f}_{\infty}}{\norm{f}_{p}}\Big)^2
\end{equation}
where $s^{-1} = 2p^{-1}-1$ and $\Lambda$ is the same function which appears in Theorem~\ref{thm:quantitative}.

Now consider the case $\tfrac43<p<2$. Let $\theta\in(0,1)$ satisfy $p^{-1}=\frac{\theta}{2}+\frac{3(1-\theta)}{4}$. 
Assume that $\norm{f}_p=1$ and write $f(x)=F(x)e^{i\varphi(x)}$ where $\varphi$ is real-valued and $F(x)\ge 0$ for all $x$.
Define $f_z(x)=F(x)^{L(z)} e^{i\varphi(x)}$ where $L:\complex\to\complex$ is the unique complex affine
function which satisfies $L(\theta)=1$, $\Re(L(z))=p/2$ whenever $\Re(z)=0$,
and $\Re(L(z))=3p/4$ whenever $\Re(z)=1$.

Then $\norm{\widehat{f_z}}_2=\norm{f_z}_2=1$ whenever $\Re(z)=0$ and, by inequality \eqref{HausYoung},
$\norm{\widehat{f_z}}_4\le\norm{f_z}_{4/3}\cdot\Lambda\Big(\frac{\norm{f_z}_{\infty}}{\norm{f_z}_{4/3}}\Big)
=\Lambda\Big(\norm{f}_\infty^{3p/4}\Big)$ whenever $\Re(z)=1$. Therefore, by the three lines lemma,
\[ \norm{\widehat{f}}_{p'}=\norm{\widehat{f_\theta}}_{p'}\le\Lambda\Big(\norm{f}_\infty^{3p/4}\Big)^\theta \]
so for any $f\in\ell^p$,
\begin{equation}
\norm{\widehat{f}}_{p'}\le\norm{f}_p\cdot\widetilde{\Lambda}\Big(\frac{\norm{f}_{\infty}}{\norm{f}_{p}}\Big)
\end{equation}
where $\widetilde{\Lambda}(t) = (\Lambda(t^{3p-4}))^\theta$ has the same character as $\Lambda$.
\qed

\end{document}